\documentclass{amsart}
\usepackage{amsmath}
\usepackage{graphicx}
\usepackage{amscd}
\usepackage{amsfonts}
\usepackage{amssymb}
\newtheorem{theorem}{Theorem}
\theoremstyle{plain}

\newtheorem{corollary}[theorem]{Corollary}

\newtheorem{definition}[theorem]{Definition}
\newtheorem{example}[theorem]{Example}

\newtheorem{lemma}[theorem]{Lemma}

\numberwithin{equation}{section}

\numberwithin{equation}{section}
\numberwithin{theorem}{section}
\begin{document}
\title[Pseudoperiodicity and the $3x+1$ Conjugacy function]{Pseudoperiodicity and the $3x+1$ Conjugacy function}
\author{Jonathan Yazinski}
\address{Dept. of Mathematics and Statistics\\
McMaster University\\
1280 Main Street West\\
Hamilton, ON L8S 4K1}
\email{yazinski@math.mcmaster.ca}

\begin{abstract}
The $3x+1$ function $T:\mathbb{N}\rightarrow\mathbb{N}$ is defined by
$T\left(  x\right)  =\frac{3x+1}{2}$ for $x$ odd and $T\left(  x\right)
=\frac{x}{2}$ for $x$ even. The function $T$ has a natural extension to the
$2$-adic integers $\mathbb{Z}_{2}$ and there is a continuous function $\Phi$
which conjugates $T$ to the $2$-adic shift map $\sigma.$ Bernstein and
Lagarias conjectured that $-1$ and $\frac{1}{3}$ are the only odd fixed points
of $\Phi.$ In this paper we investigate periodicity associated with $\Phi$, a
property of the map which is a natural extention of solenoidality. We use it
to show that there are nontrivial infinite families of $2$-adics that are not
fixed points of $\Phi.$ In particular, we prove that three sequences of
farPoints of $2$-adics are finitely pseudoperiodic, providing more evidence
supporting the $\Phi$ Fixed Point Conjecture.

\end{abstract}
\maketitle

\section{Introduction}

Let $\mathbb{Z}_{2}$ denote the ring of $2$-adic integers. If $x\in
\mathbb{Z}_{2}$ with $x=\underset{i=0}{\overset{\infty}{\sum}}d_{i}2^{i}$ with
$d_{i}\in\left\{  0,1\right\}  $, we write $x=d_{0}d_{1}d_{2}\cdots.$ Define
the $3x+1$ function $T:\mathbb{Z}_{2}\rightarrow\mathbb{Z}_{2}$ as follows.
\[
T\left(  x\right)  =\left\{
\begin{tabular}
[c]{ll}
$\frac{3x+1}{2}$ & if $x$ is odd\\
$\frac{x}{2}$ & if $x$ is even
\end{tabular}
\ \right.
\]
The famous $3x+1$ conjecture states that for every $x\in\mathbb{N}^{+}$ there
is some$\;k\in\mathbb{N}^{+}$ such that the $k$-th iterate $T^{k}\left(
x\right)  =1$, i.e. the conjecture asserts that the $T$-orbit of every
positive integer will eventually enter the cycle $\left\{  1,2\right\}  $.
Define the $2$-adic shift map $\sigma:\mathbb{Z}_{2}\rightarrow\mathbb{Z}_{2}$
by
\[
\sigma\left(  x\right)  =\left\{
\begin{tabular}
[c]{ll}
$\frac{x-1}{2}$ & if $x$ is odd\\
$\frac{x}{2}$ & if $x$ is even
\end{tabular}
\ \right.
\]
so that $\sigma\left(  d_{0}d_{1}\cdots\right)  =d_{1}d_{2}\cdots.$ Hence
$\sigma$ just removes the first digit of a $2$-adic integer\footnote{$\sigma$
is called $S$ in \cite{BL}}. There is a continuous bijective map
$\Phi:\mathbb{Z}_{2}\rightarrow\mathbb{Z}_{2}$ such that the following diagram
commutes:
\begin{equation}
\begin{tabular}
[c]{ccc}
$\mathbb{Z}_{2}$ & $\overset{\sigma}{\longrightarrow}$ & $\mathbb{Z}_{2}$\\
$\Phi\downarrow\;\;$ &  & $\Phi\downarrow\;\;$\\
$\mathbb{Z}_{2}$ & $\overset{T}{\longrightarrow}$ & $\mathbb{Z}_{2}$
\end{tabular}
\ \label{PhiCommute}
\end{equation}
so that $\Phi\circ\sigma=T\circ\Phi\;$\cite{L}$.$ The map $\Phi$ is
\textbf{solenoidal}, that is, $x\underset{2^{n}}{\equiv}y\Rightarrow
\Phi\left(  x\right)  \underset{2^{n}}{\equiv}\Phi\left(  y\right)  $
\cite{BL} (where $\underset{2^{n}}{\equiv}$ denotes congruence
$\operatorname{mod}2^{n}$). In fact, $\Phi$ induces permutations on
$\mathbb{Z}\diagup2^{n}\mathbb{Z}.$

An explicit formula for $\Phi$ was discovered by Bernstein \cite{B}, but a
convenient one is not known for its inverse. If $x=\underset{i}{\sum}2^{d_{i}
},$ with $0\leq d_{0}<d_{1}<\cdots,$ then
\begin{equation}
\Phi\left(  x\right)  =-\underset{i}{\sum}2^{d_{i}}3^{-i}\label{BrnPhiFormula}
\end{equation}
It follows from the conjugacy in (\ref{PhiCommute}) that the inverse map is
given by
\[
\Phi^{-1}\left(  x\right)  =\underset{i=0}{\overset{\infty}{\sum}}\left(
T^{i}\left(  x\right)  \;Mod\;2\right)  2^{i}
\]
and we say that $\Phi^{-1}$ maps a $2$-adic to its \textit{parity
vector}\footnote{As given in \cite{BL}, the formula for $\Phi^{-1}\left(
x\right)  $ uses $\operatorname{mod}$ instead of $Mod,$ but throughout the
course of this paper we will keep the convention that for $z\in\mathbb{Z}_{2}$
and $k\in\mathbb{N},$ $z\;Mod\;2^{k}$ will be the least $m\in\mathbb{N}$ such
that $m\underset{2^{k}}{\equiv}z.$}.

Many of the interesting properties of the $3x+1$ conjugacy function do not
lead to an apparent solution to the $3x+1$ problem \cite{BL}. The map $\Phi$
seems to inherit much of the problematic nature of $T,$ and the authors of
\cite{BL} have proposed the seemingly equally intractable

\bigskip

$\Phi$\textbf{\ Fixed Point Conjecture:}\ The only odd fixed points of $\Phi$
are $-1$ and $\frac{1}{3}.$

\bigskip

From (\ref{BrnPhiFormula}), we see that $\Phi\left(  2^{k}x\right)  =2^{k}
\Phi\left(  x\right)  $ so that every even fixed point of $\Phi$ is either $0$
or $2^{k}x$, where $x$ is an odd fixed point. There is another more general
conjecture proposed in \cite{BL} stating that for any given $2$-power, $\Phi$
will have finitely many odd periodic points whose period is that $2$-power.

\bigskip

\section{Main Results}

We now provide evidence supporting the $\Phi$ Fixed Point Conjecture by
demonstrating a periodicity associated with $\Phi$ and a related property
which we will call pseudoperiodicity. Solenoidality allows us to study the
action of $\Phi$ on the first $k$ digits of a $2$-adic; however, periodicity
will allow us to look at beginning of the infinite ``tail''\ end. First, we
define what it means for a sequence to be pseudoperiodic.

\begin{definition}
\label{defPseudo}Let $X$ be a set, and let $\left\{  a_{n}\right\}  $ be a
sequence (where $a:\mathbb{N}\rightarrow X$). For all $n\in\mathbb{N}$ we say
that the $n$-th term $a_{n}$ is \textbf{pseudoperiodic} if
\[
\exists m\in\mathbb{N\;\forall}r\mathbb{\in N}^{+}\;a_{n+mr}=a_{n}
\]
We say that the sequence $\left\{  a_{n}\right\}  $ is pseudoperiodic if all
of its terms are pseudoperiodic.
\end{definition}

We will also say that the sequence $\left\{  a_{n}\right\}  _{n\geq k}$ is
pseudoperiodic if we replace ``For all $n\in\mathbb{N}$''\ in the above
definition with ``For all $n\geq k$''. It will be convenient in some
applications to take $X\cup\left\{  \infty\right\}  ,$ where $X\subseteq
\mathbb{R}$ to be the codomain of $a$ and say that in this case $\left\{
a_{n}\right\}  $ is \textbf{finitely pseudoperiodic} when all terms $a_{n}
\neq\infty$ are pseudoperiodic.

Hence a sequence is pseudoperiodic if for every term in the sequence, there is
some period $m$ with which that term will repeat in the sequence. For example,
every periodic sequence is pseudoperiodic. For a somewhat less trivial
example, consider $b:\mathbb{N}\rightarrow\mathbb{N}$, where $b_{n}$ is
defined to be the smallest positive prime dividing $n+2.$ The
pseudoperiodicity of this sequence is reminiscent of the sieve of Eratosthenes.

Because $\Phi$ is solenoidal, ``the action of $\Phi$ on the first $k$ digits
of a $2$-adic'' is well defined. That is, the first $k$ digits of $\Phi\left(
x\right)  $ are completely determined by the first $k$ digits of $x.$
Therefore we can classify non-fixed points by their congruence class
$\operatorname{mod}2^{k}.$ In fact, we will say that $x$\textbf{\ is a fixed
point of }$\Phi\operatorname{mod}2^{k}$ when $\Phi\left(  x\right)
\underset{2^{k}}{\equiv}x.$ Consequently, whenever $x$ is a fixed point
$\operatorname{mod}2^{k},$ we can ask how many more digits $r$ are needed so
that there is no fixed point $\operatorname{mod}2^{k+r}$ having the same first
$k$ digits as $x$. This motivates the following definition.

\begin{definition}
For $k\in\mathbb{N}$ and $x\in\mathbb{N}\subseteq\mathbb{Z}_{2}$ with
$x<2^{k},$ define
\[
\operatorname{fP}\left(  x,k\right)  =\min\left\{  r\in\mathbb{N}\mid\forall
z\in\mathbb{Z}_{2},\;z\underset{2^{k}}{\equiv}x\Rightarrow\Phi\left(
z\right)  \underset{2^{k+r}}{\not \equiv }z\right\}
\]
when the minimum exists, and $\infty$ otherwise.
\end{definition}

That is, $\operatorname{fP}\left(  x,k\right)  $ is the smallest number of
digits $r$ needed so that no $2$-adic agreeing with $x$ on the first $k$
digits will be a fixed point of $\Phi\operatorname{mod}2^{k+r}$. We call
$\operatorname{fP}\left(  x,k\right)  $ the \textbf{farPoint of }
$x$\textbf{\ on the first }$k$\textbf{\ digits}.

\begin{example}
If $\Phi\left(  x\right)  \underset{2^{k}}{\not \equiv }x$ then
$\operatorname{fP}\left(  x,k\right)  =0.$ The converse is also true.
\end{example}

\begin{example}
$\operatorname{fP}\left(  0,k\right)  =\infty$ for any $k,$ since $\Phi\left(
0\right)  =0.$
\end{example}

\begin{example}
Since $\Phi\left(  1\right)  \underset{2^{2}}{\equiv}1$ and $\Phi\left(
1\right)  \underset{2^{3}}{\equiv}5$ we have that $\operatorname{fP}\left(
1,2\right)  =1.$
\end{example}

For $x\in\mathbb{Z}_{2},$ denote $L_{k}\left(  x\right)  =x\;Mod\;2^{k}$ and
$R_{k}\left(  x\right)  =\sigma^{k}\left(  x\right)  .$ $L_{k}\left(
x\right)  $ and $R_{k}\left(  x\right)  $ will be called the \textbf{(}
$k$\textbf{-)left} and \textbf{(}$k$\textbf{-)right} parts of $x,$
respectively. Thus, if $x=d_{0}d_{1}\cdots d_{k-1}d_{k}\cdots,$ then
$L_{k}\left(  x\right)  =d_{0}d_{1}\cdots d_{k-1}$ and $R_{k}\left(  x\right)
=d_{k}d_{k+1}\cdots.$ So, for example, $L_{3}\left(  \frac{1}{3}\right)  =3$
and $R_{3}\left(  \frac{1}{3}\right)  =-\frac{1}{3},$ since $\frac{1}
{3}=1101010\cdots$ and $-\frac{1}{3}=101010\cdots.$ With these definitions we
are able to state our first result.

\begin{theorem}
\label{periocityOfPhi}Fix $H\geq3$ and $q\in\mathbb{Q}\cap\mathbb{Z}_{2}.$
Then the sequence $\left\{  L_{H}R_{n}\Phi L_{n}\left(  q\right)  \right\}  $
is eventually periodic.
\end{theorem}

We call the property of $\Phi$ exhibited in this theorem the
\textit{periodicity of }$\Phi.$ We can use this property to make progress on
the Fixed Point Conjecture by showing that certain sequences of farPoints are pseudoperiodic.

\begin{theorem}
\label{pseudo1}The sequence $\left\{  \operatorname{fP}\left(  2^{n}
-1,n+1\right)  \right\}  $ is finitely pseudoperiodic.
\end{theorem}

If the $\Phi$ Fixed Point Conjecture is true, then all terms in the above
sequence will be finite (and hence pseudoperiodic) except for when $n=0,2.$
Any odd $x\in\mathbb{Z}_{2}-\left\{  \frac{1}{3},-1\right\}  $ that is a fixed
point of $\Phi$ will agree with either $-1$ or $\frac{1}{3}$ on a finite
number of digits. Direct calculation shows that such a fixed point $x$ must
agree with one of these fixed points on at least its first $100$ digits.
Theorem \ref{pseudo1} corresponds to those potential fixed points whose
initial segments agree with $-1$ on a finite number of digits. We have an
analogous theorem for $\frac{1}{3}$:

\begin{theorem}
\label{pseudo2}The sequences $\left\{  \operatorname{fP}\left(  1+\underset
{i=0}{\overset{n}{\sum}}2^{2i+1},2n+4\right)  \right\}  $ and $\left\{
\operatorname{fP}\left(  1+\underset{i=0}{\overset{n}{\sum}}2^{2i+1}
+2^{2n+2},2n+3\right)  \right\}  $ are finitely psuedoperiodic.
\end{theorem}

Note that $\frac{1}{3}=1+\underset{i=0}{\overset{\infty}{\sum}}2^{2i+1}.$
Knowing that a particular value of farPoint is finite in any of the three
sequences in Theorems \ref{pseudo1} and \ref{pseudo2} yields an infinite
family of non-fixed points of $\Phi.$ In particular, the above theorems imply

\begin{corollary}
\label{noFixedPointFamilies}If there are no fixed points of the form
$\overset{n\text{ ones}}{\overbrace{11\cdots1}0}\cdots$ (respectively,
$\overset{n\text{ ones}}{11\overbrace{010\cdots10}0}\cdots$ or $\overset
{n\text{ ones}}{11\overbrace{010\cdots1}1}\cdots$), then there will be some
$h$ such that for every $r\in N,$ there will be no fixed point of the form
$\overset{n+rh\text{ ones}}{\overbrace{1111\cdots1}0}\cdots$ (respectively,
$\overset{n+rh\text{ ones}}{11\overbrace{010\cdots10}}0\cdots$ or
$\overset{n+rh\text{ ones}}{11\overbrace{010\cdots1}}1\cdots$).
\end{corollary}

Interestingly, the $h$ in Corollary \ref{noFixedPointFamilies} turns out to be
a $2$ power in every case. Some values of farPoint are listed in Table
\ref{fPTable}, from which we obtain an immediate example of the
pseudoperiodicity from Theorem \ref{pseudo1}.

\begin{table}[ptb]
\centering
\par
$
\begin{tabular}
[c]{|l|l||}\hline
$n$ & $\operatorname{fP}_{n}$\\\hline\hline
$1$ & $1$\\\hline
$2$ & $\infty$\\\hline
$3$ & $1$\\\hline
$4$ & $7$\\\hline
$5$ & $1$\\\hline
$6$ & $10$\\\hline
$7$ & $1$\\\hline
$8$ & $29$\\\hline
$9$ & $1$\\\hline
\end{tabular}
\ \
\begin{tabular}
[c]{|l|l||}\hline
$n$ & $\operatorname{fP}_{n}$\\\hline\hline
$10$ & \multicolumn{1}{|c||}{$5$}\\\hline
$11$ & \multicolumn{1}{|c||}{$1$}\\\hline
$12$ & \multicolumn{1}{|c||}{$18$}\\\hline
$13$ & \multicolumn{1}{|c||}{$1$}\\\hline
$14$ & \multicolumn{1}{|c||}{$11$}\\\hline
$15$ & \multicolumn{1}{|c||}{$1$}\\\hline
$16$ & \multicolumn{1}{|c||}{$83$}\\\hline
$17$ & \multicolumn{1}{|c||}{$1$}\\\hline
$18$ & \multicolumn{1}{|c||}{$75$}\\\hline
\end{tabular}
\ \
\begin{tabular}
[c]{|l|l||}\hline
$n$ & $\operatorname{fP}_{n}$\\\hline\hline
$19$ & \multicolumn{1}{|c||}{$1$}\\\hline
$20$ & \multicolumn{1}{|c||}{$8$}\\\hline
$21$ & \multicolumn{1}{|c||}{$1$}\\\hline
$22$ & \multicolumn{1}{|c||}{$6$}\\\hline
$23$ & \multicolumn{1}{|c||}{$1$}\\\hline
$24$ & \multicolumn{1}{|c||}{$11$}\\\hline
$25$ & \multicolumn{1}{|c||}{$1$}\\\hline
$26$ & \multicolumn{1}{|c||}{$5$}\\\hline
$27$ & \multicolumn{1}{|c||}{$1$}\\\hline
\end{tabular}
\ \
\begin{tabular}
[c]{|l|l||}\hline
$n$ & $\operatorname{fP}_{n}$\\\hline\hline
$28$ & \multicolumn{1}{|c||}{$6$}\\\hline
$29$ & \multicolumn{1}{|c||}{$1$}\\\hline
$30$ & \multicolumn{1}{|c||}{$12$}\\\hline
$31$ & \multicolumn{1}{|c||}{$1$}\\\hline
$32$ & \multicolumn{1}{|c||}{$62$}\\\hline
$33$ & \multicolumn{1}{|c||}{$1$}\\\hline
$34$ & \multicolumn{1}{|c||}{$13$}\\\hline
$35$ & \multicolumn{1}{|c||}{$1$}\\\hline
$36$ & \multicolumn{1}{|c||}{$15$}\\\hline
\end{tabular}
\ \
\begin{tabular}
[c]{|l|l||}\hline
$n$ & $\operatorname{fP}_{n}$\\\hline\hline
$37$ & \multicolumn{1}{|c||}{$1$}\\\hline
$38$ & \multicolumn{1}{|c||}{$11$}\\\hline
$39$ & \multicolumn{1}{|c||}{$1$}\\\hline
$40$ & \multicolumn{1}{|c||}{$10$}\\\hline
$41$ & \multicolumn{1}{|c||}{$1$}\\\hline
$42$ & \multicolumn{1}{|c||}{$5$}\\\hline
$43$ & \multicolumn{1}{|c||}{$1$}\\\hline
$44$ & \multicolumn{1}{|c||}{$10$}\\\hline
$45$ & \multicolumn{1}{|c||}{$1$}\\\hline
\end{tabular}
\ \
\begin{tabular}
[c]{|l|l||}\hline
$n$ & $\operatorname{fP}_{n}$\\\hline\hline
$46$ & \multicolumn{1}{|c||}{$7$}\\\hline
$47$ & \multicolumn{1}{|c||}{$1$}\\\hline
$48$ & \multicolumn{1}{|c||}{$13$}\\\hline
$49$ & \multicolumn{1}{|c||}{$1$}\\\hline
$50$ & \multicolumn{1}{|c||}{$8$}\\\hline
$51$ & \multicolumn{1}{|c||}{$1$}\\\hline
$52$ & \multicolumn{1}{|c||}{$9$}\\\hline
$53$ & \multicolumn{1}{|c||}{$1$}\\\hline
$54$ & \multicolumn{1}{|c||}{$6$}\\\hline
\end{tabular}
\ \
\begin{tabular}
[c]{|l|l||}\hline
$n$ & $\operatorname{fP}_{n}$\\\hline\hline
$55$ & \multicolumn{1}{|c||}{$1$}\\\hline
$56$ & \multicolumn{1}{|c||}{$37$}\\\hline
$57$ & \multicolumn{1}{|c||}{$1$}\\\hline
$58$ & \multicolumn{1}{|c||}{$5$}\\\hline
$59$ & \multicolumn{1}{|c||}{$1$}\\\hline
$60$ & \multicolumn{1}{|c||}{$6$}\\\hline
$61$ & \multicolumn{1}{|c||}{$1$}\\\hline
$62$ & \multicolumn{1}{|c||}{$13$}\\\hline
$63$ & \multicolumn{1}{|c||}{$1$}\\\hline
\end{tabular}
\ \
\begin{tabular}
[c]{|l|l|}\hline
$n$ & $\operatorname{fP}_{n}$\\\hline\hline
$64$ & \multicolumn{1}{|c|}{$45$}\\\hline
$65$ & \multicolumn{1}{|c|}{$1$}\\\hline
$66$ & \multicolumn{1}{|c|}{$10$}\\\hline
$67$ & \multicolumn{1}{|c|}{$1$}\\\hline
$68$ & \multicolumn{1}{|c|}{$7$}\\\hline
$69$ & \multicolumn{1}{|c|}{$1$}\\\hline
$70$ & \multicolumn{1}{|c|}{$12$}\\\hline
$71$ & \multicolumn{1}{|c|}{$1$}\\\hline
$72$ & \multicolumn{1}{|c|}{$9$}\\\hline
\end{tabular}
\ $\caption{Values of $\operatorname{fP}_{n}=\operatorname{fP}\left(
2^{n}-1,n+1\right)  $ for $n$ from $1$ to $72$ }
\label{fPTable}
\end{table}

\begin{example}
We see that we can have no $\Phi$ fixed point of the form $1110\cdots$ from
Table \ref{fPTable}, and in this case we have that the $h$ in Corollary
\ref{noFixedPointFamilies} will be equal to $2$. So there will be no fixed
points of the form $111110\cdots$, $11111110\cdots,$ or in general
$\overset{1+2r\text{ ones}}{\overbrace{1111\cdots1}0}\cdots$ for any
$r\in\mathbb{N},$ so that there is no fixed point of $\Phi$ whose digit
expansion will have an odd number of ones followed by a zero.
\end{example}

We will see why $h=2$ is the period of this value of farPoint in the proof of
Theorem \ref{pseudo1} in the next section. Continuing the table further would
illustrate the pseudoperiodicity of all of these farPoints. For example,
$\operatorname{fP}\left(  2^{10}-1,11\right)  $ is pseudoperiodic with period
$h=2^{5-1}=16,$ as seen by the farPoint value of $5$ for $n=10,26,42,58$ in
the above table.

\section{Pseudohomomorphism and $\Phi$ Periodicity}

Here we discuss some general properties of $\Phi$ that will be used in the
proofs of our main results. If $x\in\mathbb{N\subseteq Z}_{2},$ then we define
$\alpha\left(  x\right)  =\underset{i=0}{\overset{\infty}{\sum}}d_{i} $ where
$x=d_{0}d_{1}\cdots.$ The following theorem shows that, with sufficient
restrictions, $\Phi$ acts on the right and left parts of a $2$-adic nearly as
an additive homomorphism.

\begin{theorem}
\label{pseudohomo}Let $x\in\mathbb{Z}_{2}$, $k\in\mathbb{N}$ and
$a\in\mathbb{N}$ with $a<2^{k}.$ Denote $m=\alpha\left(  a\right)  .$ Then
\[
\Phi\left(  a+2^{k}x\right)  =\Phi\left(  a\right)  +\frac{\Phi\left(
x\right)  }{3^{m}}2^{k}
\]

\end{theorem}

\begin{proof}
If $a=2^{d_{0}}+2^{d_{1}}+\cdots+2^{d_{m-1}}$ and $x=2^{d_{m}}+2^{d_{m+1}
}+\cdots,$ with $0\leq d_{0}<d_{1}<\cdots<d_{m-1}$ and $0\leq d_{m}
<d_{m+1}<\cdots,$ then by (\ref{BrnPhiFormula}),
\begin{align*}
\Phi\left(  a+2^{k}x\right)   &  =-\left(  \underset{i=0}{\overset{m-1}{\sum}
}2^{d_{i}}3^{-i}+\underset{i=0}{\overset{\infty}{\sum}}2^{d_{i+m}
+k}3^{-\left(  i+m\right)  }\right) \\
&  =-\underset{i=0}{\overset{m-1}{\sum}}2^{d_{i}}3^{-i}+\frac{2^{k}}{3^{m}
}\left(  -\underset{i=0}{\overset{\infty}{\sum}}2^{d_{i+m}}3^{-i}\right) \\
&  =\Phi\left(  a\right)  +\Phi\left(  x\right)  \frac{2^{k}}{3^{m}}
\end{align*}

\end{proof}

Thus, using the ($k$-)left and ($k$-)right notation, for any $x\in
\mathbb{Z}_{2},$
\begin{equation}
\Phi\left(  x\right)  =\Phi\left(  L_{k}\left(  x\right)  \right)  +\frac
{\Phi\left(  R_{k}\left(  x\right)  \right)  }{3^{m}}2^{k}
\end{equation}
for any $k,$ where $m=\alpha\left(  L_{k}\left(  x\right)  \right)  .$
Whenever $m$ is divisible by a positive power of $2,$ the factor $3^{m}$ acts
as $1$ for sufficiently low modulus, so that $\Phi\left(  a+2^{k}x\right)
\equiv\Phi\left(  a\right)  +\Phi\left(  x\right)  2^{k}$, as can be seen in
the following Lemma.

\begin{lemma}
\label{3power}Let $n\in\mathbb{N}^{+}$. Then $3^{2^{n}}\underset{2^{n+2}
}{\equiv}1.$
\end{lemma}

\begin{proof}
The proof will be by induction. We have that $3^{2}\underset{2^{3}}{\equiv}1.$
For $n\in\mathbb{N}^{+},$ assume $3^{2^{n}}\underset{2^{n+2}}{\equiv}1. $ Then
for some $k\in\mathbb{Z},$ $2^{n+2}k+1=3^{2^{n}}$, and hence $\left(
2^{n+2}\right)  ^{2}k^{2}+2\cdot2^{n+2}k+1=3^{2^{n+1}}.$ The result follows.
\end{proof}

Recall that $\mathbb{Z}_{2}$ contains a subring of the rational numbers,
namely those with odd denominator when in reduced form. Define $\mathbb{Q}
_{odd}=\mathbb{Q\cap Z}_{2}$ to be this subring, as in \cite{MY}. Then any
$q\in\mathbb{Q}_{odd}$ must be of the form $q=d_{0}d_{1}d_{2}\cdots
d_{k-1}\overline{d_{k}d_{k+1}\cdots d_{k+v-1}}$. Hence for any $q\in
\mathbb{Q}_{odd},$ there will be some $a,k,b,v\in\mathbb{N},$ with $a<2^{k}$
and $b<2^{v}$ such that $q=a+2^{k}\left(  \underset{i=0}{\overset{\infty}
{\sum}}b2^{vi}\right)  .$ To simplify this, we introduce the following
notation. For $b,v,t\in\mathbb{N},$ we denote $\bar{b}_{v,t}=b\underset
{i=0}{\overset{t-1}{\sum}}2^{vi}$ and $\bar{b}_{v,\infty}=b\underset
{i=0}{\overset{\infty}{\sum}}2^{vi}.$ In this notation, $q=a+2^{k}\bar
{b}_{v,\infty}.$ There is a strong relationship between the action of $\Phi$
on a rational $2$-adic and on a $2$-adic that results from a truncation after
a finite number of its repeating units, as is given by the following theorem.

\begin{theorem}
\label{infinitevsRational}Let $q\in\mathbb{Q}_{odd}$ and $a,k,b,v\in
\mathbb{N}$ such that $a<2^{k},b<2^{v},$ and $q=a+2^{k}\bar{b}_{v,\infty}.$
Define $m=\alpha\left(  b\right)  .$ Let $t\in\mathbb{N}^{+}.$ Then

\[
\Phi\left(  a+2^{k}\bar{b}_{v,t}\right)  =\Phi\left(  q\right)  -\frac
{\Phi\left(  \bar{b}_{v,\infty}\right)  }{3^{mt+\alpha\left(  a\right)  }
}2^{k+tv}
\]

\end{theorem}

\begin{proof}
Clearly $L_{k+tv}\left(  q\right)  =a+2^{k}\bar{b}_{v,t}$ and $R_{k+tv}\left(
q\right)  =\bar{b}_{v,\infty},$ so by again Theorem \ref{pseudohomo},
\[
\Phi\left(  q\right)  =\Phi\left(  a+2^{k}\bar{b}_{v,t}+2^{k+tv}\bar
{b}_{v,\infty}\right)  =\Phi\left(  a+2^{k}\bar{b}_{v,t}\right)  +\frac
{\Phi\left(  \bar{b}_{v,\infty}\right)  }{3^{mt+\alpha\left(  a\right)  }
}2^{k+tv}
\]

\end{proof}

We also state the analogous theorem for $2$-adics whose digits are strictly repeating.

\begin{corollary}
\label{infinitevs}Let $v\in\mathbb{N}^{+},t\in\mathbb{N}$ and let
$b\in\mathbb{N}$ such that $b<2^{k}.$ Define $m=\alpha\left(  b\right)  .$ Then

\[
\Phi\left(  \bar{b}_{v,t}\right)  =\Phi\left(  \bar{b}_{v,\infty}\right)
-\frac{\Phi\left(  \bar{b}_{v,\infty}\right)  }{3^{mt}}2^{tv}
\]

\end{corollary}

\begin{proof}
Let $q=\bar{b}_{v,\infty}$ in Theorem \ref{infinitevsRational}
\end{proof}

Now we present the $\Phi$ periodicity result for rational $2$-adics. It is
well known that $\Phi\left(  \mathbb{Q}_{odd}\right)  \subseteq\mathbb{Q}
_{odd},$ since if a $2$-adic has a rational parity vector, then it is the
solution of a linear equation with rational coefficients. Thus $\Phi$ of any
rational number will have eventually repeating digits. The following result,
which does not take into account fixed points, can be used to obtain the
desired pseudoperiodicity of the farPoint sequences.

\begin{theorem}
\label{mainthmRational}Let $q\in\mathbb{Q}_{odd},$ and $a,k,b,v\in\mathbb{N}$
such that $a<2^{k},b<2^{v}$ and $q=a+2^{k}\bar{b}_{v,\infty}.$ Let
\[
\Phi\left(  q\right)  =r_{0}r_{1}\cdots r_{u-1}\overline{r_{u}r_{u+1}\cdots
r_{u+p-1}}
\]
and choose $t$ such that $k+tv\geq u.$ Then for all positive integers $H,$
\[
R_{k+tv}\Phi L_{k+tv}\left(  q\right)  \underset{2^{H+2}}{\equiv}R_{k+\left(
t+p2^{H}\right)  v}\Phi L_{k+\left(  t+p2^{H}\right)  v}\left(  q\right)
\]

\end{theorem}

\begin{proof}
Define $m=\alpha\left(  b\right)  .$ From Theorem \ref{infinitevsRational} we
see that
\begin{align*}
R_{k+tv}\left(  \Phi\left(  q\;Mod\;2^{k+tv}\right)  \right)   &
=R_{k+tv}\left(  \Phi\left(  a+2^{k}\bar{b}_{v,t}\right)  \right) \\
&  =R_{k+tv}\left(  \Phi\left(  q\right)  -\frac{\Phi\left(  \bar{b}
_{v,\infty}\right)  }{3^{mt+\alpha\left(  a\right)  }}2^{k+tv}\right)
\end{align*}
and also that
\begin{align*}
R_{k+\left(  t+p2^{H}\right)  v}\left(  \Phi\left(  q\;Mod\;2^{k+\left(
t+p2^{H}\right)  v}\right)  \right)   &  =R_{k+\left(  t+p2^{H}\right)
v}\left(  \Phi\left(  a+2^{k}\bar{b}_{v,t+p2^{H}}\right)  \right) \\
&  =R_{k+\left(  t+p2^{H}\right)  v}\left(  \Phi\left(  q\right)  -\frac
{\Phi\left(  \bar{b}_{v,\infty}\right)  }{3^{m\left(  t+p2^{H}\right)
+\alpha\left(  a\right)  }}2^{k+\left(  t+p2^{H}\right)  v}\right)
\end{align*}
But by Lemma \ref{3power},
\[
\frac{\Phi\left(  \bar{b}_{v,\infty}\right)  }{3^{mt+\alpha\left(  a\right)
}}\underset{2^{H+2}}{\equiv}\frac{\Phi\left(  \bar{b}_{v,\infty}\right)
}{3^{mt+\alpha\left(  a\right)  }}\frac{1}{\left(  3^{2^{H}}\right)  ^{mp}
}=\frac{\Phi\left(  \bar{b}_{v,\infty}\right)  }{3^{m\left(  t+p2^{H}\right)
+\alpha\left(  a\right)  }}
\]
Now, since $k+tv\geq u,$
\[
R_{k+tv}\left(  \Phi\left(  q\right)  \right)  =R_{k+\left(  t+p2^{H}\right)
v}\left(  \Phi\left(  q\right)  \right)
\]
Since for any $y,z\in\mathbb{Z}_{2}$ and $c\in\mathbb{N},$
\begin{equation}
R_{c}\left(  y+2^{c}z\right)  =R_{c}\left(  y\right)
+z\label{cute2adicArithmetic}
\end{equation}
it follows that
\[
\begin{tabular}
[c]{ll}
$R_{k+tv}\left(  \Phi\left(  q\right)  -\frac{\Phi\left(  \bar{b}_{v,\infty
}\right)  }{3^{mt+\alpha\left(  a\right)  }}2^{k+tv}\right)  $ &
$\underset{2^{H+2}}{\equiv}R_{k+tv}\Phi\left(  q\right)  -\frac{\Phi\left(
\bar{b}_{v,\infty}\right)  }{3^{mt+\alpha\left(  a\right)  }}$\\
& $\underset{2^{H+2}}{\equiv}R_{k+\left(  t+p2^{H}\right)  v}\Phi\left(
q\right)  -\frac{\Phi\left(  \bar{b}_{v,\infty}\right)  }{3^{m\left(
t+p2^{H}\right)  +\alpha\left(  a\right)  }}$\\
& $\underset{2^{H+2}}{\equiv}R_{k+\left(  t+p2^{H}\right)  v}\left(
\Phi\left(  q\right)  -\frac{\Phi\left(  \bar{b}_{v,\infty}\right)
}{3^{m\left(  t+p2^{H}\right)  +\alpha\left(  a\right)  }}2^{k+\left(
t+p2^{H}\right)  v}\right)  $
\end{tabular}
\]
where the first and third lines in particular follow from
(\ref{cute2adicArithmetic}), and the result follows.
\end{proof}

There is a similar, simpler, result for strictly repeating $2$-adics:

\begin{corollary}
\label{mainthm}Let $v\in\mathbb{N}^{+},$ and $b\in\mathbb{N}$ such that
$b<2^{v}.$ Define $m=\alpha\left(  b\right)  .$ Let
\[
\Phi\left(  \bar{b}_{v,\infty}\right)  =r_{0}r_{1}\cdots r_{u-1}
\overline{r_{u}r_{u+1}\cdots r_{u+p-1}}
\]
\newline and choose $t\in\mathbb{N}$ such that $tv\geq u.$ Then for all
positive integers $H,$
\[
R_{tv}\left(  \Phi\left(  \bar{b}_{v,t}\right)  \right)  \underset{2^{H+2}
}{\equiv}R_{\left(  t+p2^{H}\right)  v}\left(  \Phi\left(  \bar{b}
_{v,t+p2^{H}}\right)  \right)
\]

\end{corollary}

\begin{proof}
Let $q=\bar{b}_{v,\infty}$ in Theorem \ref{mainthmRational}.
\end{proof}

If $q$ is a fixed point of $\Phi,$ then $k=u$ and $v=p$ in Theorem
\ref{mainthmRational}. We are in a position to prove Theorem
\ref{periocityOfPhi}.

\section{Proofs of Main Results}

\begin{proof}
[Proof of Theorem \ref{periocityOfPhi}]Let $q\in\mathbb{Q}_{odd}$ and
$a,k,b,v\in\mathbb{N}$ such that $a<2^{k},b<2^{v},$ and $q=a+2^{k}\bar
{b}_{v,\infty}$ and let $H\geq3$ be fixed. Let $t$ be the smallest positive
natural number such that $k+tv\geq u, $ where $\Phi\left(  q\right)
=r_{0}r_{1}\cdots r_{u-1}\overline{r_{u}r_{u+1}\cdots r_{u+p-1}}.$ Let $n\geq
k+tv,$ and define $l=n-\left(  k+tv\right)  .$ By Theorem
\ref{mainthmRational} (using $k+l$ for the value of $k$ in the theorem),
\[
\begin{tabular}
[c]{ll}
$R_{n}\Phi L_{n}\left(  q\right)  $ & $=R_{k+l+tv}\Phi L_{k+l+tv}\left(
q\right)  $\\
& $\underset{2^{H}}{\equiv}R_{k+l+\left(  t+p2^{H-2}\right)  v}\Phi
L_{k+l+\left(  t+p2^{H-2}\right)  v}\left(  q\right)  $\\
& $=R_{k+l+tv+p2^{H-2}v}\Phi L_{k+l+tv+p2^{H-2}v}\left(  q\right)  $\\
& $=R_{n+p2^{H-2}v}\Phi L_{n+p2^{H-2}v}\left(  q\right)  $
\end{tabular}
\]
Hence $L_{H}R_{n}\Phi L_{n}\left(  q\right)  =L_{H}R_{n+p2^{H-2}v}\Phi
L_{n+p2^{H-2}v}\left(  q\right)  .$ Therefore the sequence $\left\{
L_{H}R_{n}\Phi L_{n}\left(  q\right)  \right\}  _{n\geq k+tv}$ is periodic.
\end{proof}

From Corollary \ref{mainthm}, the result that $\left\{  \operatorname{fP}
\left(  2^{n}-1,n+1\right)  \right\}  $ is pseudoperiodic follows. To see
this, we will consider $2$-adics whose first $n$ digits agree with those of
the fixed point $-1.$

\begin{proof}
[Proof of Theorem \ref{pseudo1}]Let $n\in\mathbb{N}$ and assume that the
farPoint of $2^{n}-1$ on the first $n+1$ digits is finite. Let
$f=\operatorname{fP}\left(  2^{n}-1,n+1\right)  $ and $H\geq\max\left\{
1,f-1\right\}  .$ So there is no $z\in\mathbb{Z}_{2}$ such that $z\underset
{2^{n+1}}{\equiv}2^{n}-1$ and $\Phi\left(  z\right)  \underset{2^{n+1+f}
}{\equiv}z,$ by definition of farPoint. Let $y\in\mathbb{Z}_{2}$ be arbitrary
and define
\[
z=2^{n}-1+2^{n+1}y\text{ and }z^{\prime}=2^{n+2^{H}}-1+2^{n+1+2^{H}}y
\]
Then $z$ is not a fixed point of $\Phi\operatorname{mod}2^{n+1+f},$ and by
Theorem \ref{pseudohomo}, $\Phi\left(  z\right)  =\Phi\left(  2^{n}-1\right)
+\frac{\Phi\left(  2y\right)  }{3^{n}}2^{n}.$ Similarly, $\Phi\left(
z^{\prime}\right)  =\Phi\left(  2^{n+2^{H}}-1\right)  +\frac{\Phi\left(
2y\right)  }{3^{n+2^{H}}}2^{n+2^{H}}.$ Furthermore, since $3^{2^{H}}
\underset{2^{H+2}}{\equiv}1$ by Lemma \ref{3power} and $R_{n+2^{H}}\Phi\left(
2^{n+2^{H}}-1\right)  \underset{2^{H+2}}{\equiv}R_{n}\Phi\left(
2^{n}-1\right)  $ by Corollary \ref{mainthm}, letting $b=v=1,$ again using
(\ref{cute2adicArithmetic}),
\[
\begin{tabular}
[c]{ll}
$R_{n+2^{H}}\Phi\left(  z^{\prime}\right)  $ & $\underset{2^{H+2}}{\equiv
}R_{n+2^{H}}\Phi\left(  2^{n+2^{H}}-1\right)  +\frac{\Phi\left(  2y\right)
}{3^{m+2^{H}}}$\\
& $\underset{2^{H+2}}{\equiv}R_{n}\Phi\left(  2^{n}-1\right)  +\frac
{\Phi\left(  2y\right)  }{3^{m}}$\\
& $\underset{2^{H+2}}{\equiv}R_{n}\Phi\left(  z\right)  $
\end{tabular}
\]
Furthermore, it is clear from how we have defined $z$ and $z^{\prime}$ that
$R_{n}\left(  z\right)  =2y=R_{n+2^{H}}\left(  z^{\prime}\right)  .$ Now, by
definition of farPoint, $R_{n}\left(  \Phi\left(  z\right)  \right)
\underset{2^{f+1}}{\not \equiv }R_{n}\left(  z\right)  ,$ and hence we have
that
\[
R_{n+2^{H}}\left(  \Phi\left(  z^{\prime}\right)  \right)  \underset{2^{H+2}
}{\equiv}R_{n}\left(  \Phi\left(  z\right)  \right)  \underset{2^{f+1}
}{\not \equiv }R_{n}\left(  z\right)  \underset{2^{H+2}}{\equiv}R_{n+2^{H}
}\left(  z^{\prime}\right)
\]
Thus $\operatorname{fP}\left(  2^{n+2^{H}}-1,n+1+2^{H}\right)  \leq f.$

To show the other inequality, we choose $y\in\mathbb{Z}_{2}$ so that if
$z=2^{n}-1+2^{n+1}y,$ then $\Phi\left(  z\right)  \underset{2^{n+f}}{\equiv}z.
$ We know that such a $y$ exists by definition of farPoint. As before, we will
define $z^{\prime}=2^{n+2^{H}}-1+2^{n+1+2^{H}}y$ for this choice of $y, $ and
from the above discussion we know that $R_{n+2^{H}}\Phi\left(  z^{\prime
}\right)  \underset{2^{f}}{\equiv}R_{n}\Phi\left(  z\right)  $ and
$R_{n}\left(  z\right)  \underset{2^{f}}{\equiv}R_{n+2^{H}}\left(  z^{\prime
}\right)  $ since $f\leq H+1.$ However, since $\Phi\left(  z\right)
\underset{2^{n+f}}{\equiv}z,$ now $R_{n}\left(  \Phi\left(  z\right)  \right)
\underset{2^{f}}{\equiv}R_{n}\left(  z\right)  ,$ and therefore
\[
R_{n+2^{H}}\left(  \Phi\left(  z^{\prime}\right)  \right)  \underset{2^{f}
}{\equiv}R_{n}\left(  \Phi\left(  z\right)  \right)  \underset{2^{f}}{\equiv
}R_{n}\left(  z\right)  \underset{2^{f}}{\equiv}R_{n+2^{H}}\left(  z^{\prime
}\right)
\]
Clearly, $\Phi\left(  z^{\prime}\right)  \underset{2^{n+2^{H}}}{\equiv
}z^{\prime}$ by solenoidality since $\Phi\left(  -1\right)  =-1,$ and
therefore $\operatorname{fP}\left(  2^{n+2^{H}}-1,n+1+2^{H}\right)
=\operatorname{fP}\left(  2^{n}-1,n+1\right)  .$ This completes the proof.
\end{proof}

The proof of Theorem \ref{pseudo2} is similar, but relies instead upon Theorem
\ref{mainthmRational}.

\begin{proof}
[Proof of Theorem \ref{pseudo2}]This time we will work with those $2$-adic
initial segments agreeing nontrivially with $\frac{1}{3}=\Phi\left(  \frac
{1}{3}\right)  ,$ and so we apply Theorem \ref{mainthmRational} with
$q=\frac{1}{3}.$ Once again, for $n\in\mathbb{N},$ assume that
$f=\operatorname{fP}\left(  1+\left(  \underset{i=0}{\overset{n}{\sum}
}2^{2i+1}\right)  +2^{2n+2},2n+3\right)  $ is finite and let $H\geq
\max\left\{  1,f-1\right\}  .$ Note that $1+\underset{i=0}{\overset{n}{\sum}
}2^{2i+1}\equiv\frac{1}{3}\;Mod\;2^{2n+2},$ and that there is no
$z\in\mathbb{Z}_{2}$ with $z\underset{2^{2n+3}}{\equiv}1+\left(
\underset{i=0}{\overset{n}{\sum}}2^{2i+1}\right)  +2^{2n+2}$ and $\Phi\left(
z\right)  \underset{2^{2n+3+f}}{\equiv}z.$ Let $y\in\mathbb{Z}_{2}$ and
define
\begin{align*}
z  &  =1+\left(  \underset{i=0}{\overset{n}{\sum}}2^{2i+1}\right)
+2^{2n+2}+2^{2n+3}y\\
z^{\prime}  &  =1+\left(  \underset{i=0}{\overset{n+2^{H+1}}{\sum}}
2^{2i+1}\right)  +2^{2\left(  n+2^{H+1}\right)  +2}+2^{2\left(  n+2^{H+1}
\right)  +3}y
\end{align*}
Here $z$ is not a fixed point of $\Phi\operatorname{mod}2^{2n+3+f}.$ By
Theorem \ref{pseudohomo},
\[
\Phi\left(  z\right)  =\Phi\left(  1+\left(  \underset{i=0}{\overset{n}{\sum}
}2^{2i+1}\right)  \right)  +\frac{\Phi\left(  1+2y\right)  }{3^{n+2}}2^{2n+2}
\]
and also
\[
\Phi\left(  z^{\prime}\right)  =\Phi\left(  1+\left(  \underset{i=0}
{\overset{n+2^{H+1}}{\sum}}2^{2i+1}\right)  \right)  +\frac{\Phi\left(
1+2y\right)  }{3^{n+2^{H+1}+2}}2^{2\left(  n+2^{H+1}\right)  +2}
\]
Again, $3^{2^{H+1}}\underset{2^{H+2}}{\equiv}1$ by Lemma \ref{3power}, and
\[
R_{2+2n+2^{H+2}}\Phi\left(  1+\left(  \underset{i=0}{\overset{n+2^{H+1}}{\sum
}}2^{2i+1}\right)  +2^{2\left(  n+2^{H+1}\right)  +2}\right)  \underset
{2^{H+2}}{\equiv}R_{2+2n}\Phi\left(  1+\left(  \underset{i=0}{\overset{n}
{\sum}}2^{2i+1}\right)  +2^{2n+2}\right)
\]
by Theorem \ref{mainthmRational} with $b=k=v=2$ and $a=3.$ Therefore,
\[
\begin{tabular}
[c]{ll}
$R_{2+2n+2^{H+2}}\left(  \Phi\left(  z^{\prime}\right)  \right)  $ &
$\underset{2^{H+2}}{\equiv}R_{2+2n+2^{H+2}}\Phi\left(  1+\left(
\underset{i=0}{\overset{n+2^{H+1}}{\sum}}2^{2i+1}\right)  +2^{2\left(
n+2^{H+1}\right)  +2}\right)  +\frac{\Phi\left(  1+2y\right)  }{3^{n+2^{H+1}
+2}}$\\
& $\underset{2^{H+2}}{\equiv}R_{2+2n}\Phi\left(  1+\left(  \underset
{i=0}{\overset{n}{\sum}}2^{2i+1}\right)  +2^{2n+2}\right)  +\frac{\Phi\left(
1+2y\right)  }{3^{n+2}}$\\
& $\underset{2^{H+2}}{\equiv}\Phi\left(  z\right)  $
\end{tabular}
\]
by (\ref{cute2adicArithmetic}). It is again clear from how $z$ and $z^{\prime
}$ are defined that $R_{2+2n}\left(  z\right)  =1+2y=R_{2+2n+2^{H+2}}\left(
z^{\prime}\right)  $. By definition of farPoint, $R_{2+2n}\left(  \Phi\left(
z\right)  \right)  \underset{2^{f+1}}{\not \equiv }R_{2+2n}\left(  z\right)
,$ and hence
\[
R_{2+2n+2^{H+2}}\left(  \Phi\left(  z^{\prime}\right)  \right)  \underset
{2^{H+2}}{\equiv}R_{2+2n}\left(  \Phi\left(  z\right)  \right)  \underset
{2^{f+1}}{\not \equiv }R_{2+2n}\left(  z\right)  \underset{2^{H+2}}{\equiv
}R_{2+2n+2^{H+2}}\left(  z^{\prime}\right)
\]
So we have shown that $\operatorname{fP}\left(  1+\left(  \underset
{i=0}{\overset{n+2^{H+1}}{\sum}}2^{2i+1}\right)  +2^{2\left(  n+2^{H+1}
\right)  +2},2\left(  n+2^{H+1}\right)  +3\right)  \leq f.$ Furthermore we can
choose $y\in\mathbb{Z}_{2}$ so that if $z=1+\left(  \underset{i=0}{\overset
{n}{\sum}}2^{2i+1}\right)  +2^{2n+2}+2^{2n+3}y,$ then $\Phi\left(  z\right)
\underset{2^{2n+f+2}}{\equiv}z.$ Similarly, we define $z^{\prime}=1+\left(
\underset{i=0}{\overset{n+2^{H+1}}{\sum}}2^{2i+1}\right)  +2^{2\left(
n+2^{H+1}\right)  +2}+2^{2\left(  n+2^{H+1}\right)  +3}y$ for this same $y,$
and we can again obtain $R_{2+2n+2^{H+2}}\left(  \Phi\left(  z^{\prime
}\right)  \right)  \underset{2^{f}}{\equiv}R_{2+2n}\left(  \Phi\left(
z\right)  \right)  $ and $R_{2+2n}\left(  z\right)  \underset{2^{f}}{\equiv
}R_{2+2n+2^{H+2}}\left(  z^{\prime}\right)  $ since $f\leq H+1$. But since
$\Phi\left(  z\right)  \underset{2^{2n+f+2}}{\equiv}z,$ we have $R_{2+2n}
\left(  \Phi\left(  z\right)  \right)  \underset{2^{f}}{\equiv}R_{2+2n}\left(
z\right)  ,$ and so
\[
R_{2+2n+2^{H+2}}\left(  \Phi\left(  z^{\prime}\right)  \right)  \underset
{2^{f}}{\equiv}R_{2+2n}\left(  \Phi\left(  z\right)  \right)  \underset{2^{f}
}{\equiv}R_{2+2n}\left(  z\right)  \underset{2^{f}}{\equiv}R_{2+2n+2^{H+2}
}\left(  z^{\prime}\right)
\]
And since $\frac{1}{3}$ is a fixed point of $\Phi,$ $\Phi\left(  z^{\prime
}\right)  \equiv z^{\prime}\left(  \operatorname{mod}2+2n+2^{H+2}\right)  $ by
solenoidality, and thus we have shown that $\left\{  \operatorname{fP}\left(
1+\left(  \underset{i=0}{\overset{n}{\sum}}2^{2i+1}\right)  +2^{2n+2}
,2n+3\right)  \right\}  $ is finitely pseudoperiodic. The other half of the
proof, showing that the sequence $\left\{  \operatorname{fP}\left(
1+\underset{i=0}{\overset{n}{\sum}}2^{2i+1},2n+4\right)  \right\}  $ is
finitely pseudoperiodic is similar, also relying upon Theorem
\ref{mainthmRational}.
\end{proof}

Using the periodicity of $\Phi,$ once we have determined that there is one
$x\in\mathbb{Z}_{2}$ that is not a fixed point of $\Phi\operatorname{mod}
2^{k}$ for some $k,$ we obtain infinite families of $2$-adics that are not
fixed points of $\Phi\operatorname{mod}2^{k+mr}.$ While the applicability of
the periodicity of $\Phi$ to the Fixed Points Conjecture is clear, it is not
clear how it can be used to aid in the solution to the original $3x+1$
problem. One potentially interesting area for future study would be to apply
this property to the map $\Omega=\Phi\circ V\circ\Phi^{-1},$ where $V\left(
x\right)  =-1-x$ \cite{MY}. This map makes an intimate connection between
$\Phi$ and the main conjectures associated with the $3x+1$ problem. Such
investigations may shed light on the $3x+1$ problem itself.

\section{Acknowledgments}

This project was funded by a research grant from the University of Scranton.
Many thanks to Kenneth Monks for his many helpful comments, and in particular
for his contributions which led to simpler proofs of Theorems
\ref{infinitevsRational} and \ref{mainthmRational}.

\end{document}